\documentclass[10pt]{amsart}
\usepackage{amssymb}
\usepackage{amsmath}
\usepackage[english]{babel}
\usepackage{color}

\def\thebibliography#1{\section*{References}\list
  {[\arabic{enumi}]}{\settowidth\labelwidth{[#1]}\leftmargin\labelwidth
    \advance\leftmargin\labelsep
    \usecounter{enumi}}
    \def\newblock{\hskip .11em plus .33em minus -.07em}
    \sloppy
    \sfcode`\.=1000\relax}

\newtheorem{lemma}{Lemma}[section]
\newtheorem{corollary}[lemma]{Corollary}
\newtheorem{theorem}[lemma]{Theorem}
\newtheorem{remark}[lemma]{Remark}

\newtheorem{proposition}[lemma]{Proposition}

\newcommand{\R}{\mathbb{R}}

\newcommand{\N}{\mathbb{N}}

\newcommand{\z}{{\mathbb{R}^N}}

\newcommand{\loc}{\emph{loc\,}}
\date{}

\begin{document}
\title[Weighted Hardy's inequalities]{Weighted Hardy's inequalities and Kolmogorov-type operators}
\author[A. Canale]{A. Canale}
\address{Dipartimento di Ingegneria dell'Informazione, Ingegneria Elettrica e Ma\-te\-matica Applicata, Universit\'a degli
Studi di Salerno, Via Giovanni Paolo II, 132, I 84084 FISCIANO (Sa), Italy.}
\email{acanale@unisa.it}
\author[F. Gregorio]{F. Gregorio}
\address{Facult\"at f\"ur Mathematik und Informatik, FernUniversit\"at in Hagen, Universit\"atst. 11, 58084 Hagen, Germany.}
\email{federica.gregorio@fernuni-hagen.de}
\author[A. Rhandi]{A. Rhandi}
\thanks{The authors are members of the Gruppo Nazionale per l'Analisi Matematica, la Probabilit\'a e le loro Applicazioni (GNAMPA) of the Istituto Nazionale di Alta Matematica (INdAM)}
\address{Dipartimento di Ingegneria dell'Informazione, Ingegneria Elettrica e Ma\-te\-matica Applicata, Universit\'a degli
Studi di Salerno, Via Giovanni Paolo II, 132, I 84084 FISCIANO (Sa), Italy.}
\email{arhandi@unisa.it}
\author[C. Tacelli]{C. Tacelli}
\address{Dipartimento di Ingegneria dell'Informazione, Ingegneria Elettrica e Ma\-te\-matica Applicata, Universit\'a degli
Studi di Salerno, Via Giovanni Paolo II, 132, I 84084 FISCIANO (Sa), Italy.}
\email{ctacelli@unisa.it}

\begin{abstract}
We give general conditions to state the weighted Hardy inequality
\begin{equation*}
c\int_{\R^N}\frac{\varphi^2}{|x|^2}d\mu\leq\int_{\R^N}|\nabla \varphi |^2
d\mu+C\int_{\R^N} \varphi^2d\mu,\quad \varphi\in C_c^{\infty}(\R^N),\,c\leq c_{0,\mu},
\end{equation*}
with respect to a probability measure $d\mu$. Moreover, the optimality of the constant $c_{0,\mu}$ is given.
The inequality is related to the following Kolmogorov equation perturbed by a singular potential
\begin{equation*}
Lu+Vu=\left(\Delta u+\frac{\nabla \mu}{\mu}\cdot \nabla u\right)+\frac{c}{|x|^2}u
\end{equation*}
for which  the existence of  positive solutions to the corresponding parabolic problem can be investigated.
The hypotheses on $d\mu$ allow the drift term to be of type $\frac{\nabla \mu}{\mu}=
-|x|^{m-2}x$ with $m> 0$.
\end{abstract}
\maketitle

\section{Introduction}
We denote by $L$ be the Kolmogorov operator
\begin{equation*}
Lu=\Delta u+\frac{\nabla \mu}{\mu}\cdot\nabla u
\end{equation*}
defined on smooth functions. In the standard setting one considers $\mu \in C_{loc}^{1,\alpha}\left( \R^N \right)$
for some $\alpha\in (0,1)$ and $\mu(x)>0$ for all $x\in \R^N$. In this case
the elliptic operator $L$ has coefficients belonging to
$C_{loc}^{\alpha}\left( \R^N \right)$. Therefore, one
can associate to $L$ a semigroup $\{T(t)\}_{t\geq 0}$
(not necessary strongly continuous)  in the space of bounded continuous functions, cf. \cite{ber-lor}.
Considering the measure $d\mu=\mu (x)dx$
and the weighted space $L^2_\mu:=L^2(\R^N,d\mu)$,
the operator $L$ can also be defined via the bilinear form
\[
a_\mu(u,v)=\int _{\R^N}\nabla u\cdot \nabla \overline{v}\, d\mu
\]
on $H^1_\mu=H^1(\R^N,d\mu)$, the Sobolev space of functions whose weak derivatives belong to $L^2_\mu.$
Indeed,
integrating by parts we get
\[
a_\mu(u,v)=-\int_{\R^N}Lu\overline{v}\, d\mu,
\qquad u,v\in C_c^{\infty}(\R^N).
\]
In particular,  $\int_{\R^N}Lu\,d\mu=0$ for every $u\in C_c^{\infty}(\R^N)$.
Then $d\mu$ is an invariant measure of $\{T(t)\}$ and hence
$\{T(t)\}$ can be
extended  to a positivity preserving and analytic $C_0$-semigroup
 on $L^2_\mu$ (see for example \cite{ber-lor}).

Recently in \cite{gol-gol-rhan} and \cite{gol-gol-rhan2}, a special class of operators of type
$L$ perturbed by the inverse square potential $V(x)=\frac{c}{|x|^2}$ was considered and the associated
evolution equation
$$(P_{V})\quad \left\{\begin{array}{ll}
\partial_tu(x,t)=Lu(x,t)+V(x)u(x,t),\quad \,x\in {\mathbb R}^N, t>0,\\
u(\cdot ,0)=u_0\geq 0\in L^2_\mu,
\end{array}
\right. $$
was studied.

It is well known that the potential $\frac{c}{|x|^2},\,c>0,$ is highly singular in the sense that it belongs to a borderline case where the strong
maximum principle and Gaussian bounds fail, cf. \cite{Aronson}. Moreover, it is not in the Kato class potentials.
If $V\leq \frac{C}{|x|^{2-\varepsilon}}$,
then the initial value problem is well-posed. But for $\varepsilon=0$ the problem
$\frac{\partial u}{\partial t}=\Delta u+V$ may not have positive solutions.
In \cite{bar-gol} Baras and Goldstein show that
the evolution equation associated to $\Delta+V$ admits
a unique positive solution
if $c\leq c_0(N):=\left( \frac{N-2}{2} \right)^{2}$ and no positive solutions exist if $c>c_0(N)$ (see also \cite{cab-mar}).
When it exists, the solution is
exponentially bounded, on the contrary,  if $c>c_0(N)$, there is the so called
{\it instantaneous blowup} phenomena.

Replacing the Laplacian by the Kolmogorov operator $L$  a similar behaviour was obtained
in \cite{gol-gol-rhan}.
The result was given using a relation between the weak
solution of $(P_V)$
and the {\it bottom of the spectrum}
of the operator $-(L+V)$
\begin{equation*}
\lambda_1(L+V):=\inf_{\varphi \in H^1_\mu\setminus \{0\}}
\left(\frac{\int_{{\mathbb R}^N}|\nabla \varphi |^2\,d\mu
-\int_{{\mathbb R}^N}V\varphi^2\,d\mu}{\int_{{\mathbb R}^N}\varphi^2\,d\mu}
\right).
\end{equation*}
Cabr\'e and Martel in \cite{cab-mar} show that the boundedness  of $\lambda_1(\Delta +V)$ is a necessary and
sufficient condition for the existence of positive and exponentially bounded in time
solutions to the associated initial value problem.
This result was extended in \cite{gol-gol-rhan} to the operator $L+V$.

For Ornstein-Uhlenbeck type operators $Lu=\Delta u-\sum_{i=1}^n A(x-a_i)\cdot \nabla u$,
$a_i\in \R^N$, $i=1\dots n$,
perturbed by multipolar inverse square potentials a weighted multipolar Hardy inequality
and related existence and nonexistence results were stated in \cite{can-pap}.
In such a case the invariant measure for these operators is
$d\mu=\mu_A(x)dx=Ke^{-\frac{1}{2}\sum_{i=1}^n\langle A(x-a_i),x-a_i\rangle}dx$.

Assuming that $0<\mu \in C^{1+\alpha}_{loc}({\mathbb R}^N)$
is a probability density on ${\mathbb R}^N$ we recall the following result, see \cite[Theorem 2.1]{gol-gol-rhan}.
\begin{theorem}\label{thm11} If $0\le V\in
L_{loc}^1({\mathbb R}^N)$. Then the following hold:
\begin{enumerate}
\item[(i)] If $\lambda_1(L+V)>-\infty$, then there exists a
nonnegative weak solution \\$u\in C([0,\infty),L^2_\mu)$ of $(P_V)$
satisfying
\begin{equation}\label{eq22}
\|u(t)\|_{L^2_\mu}\le Me^{\omega t}\|u_0\|_{L^2_\mu},\quad t\ge
0
\end{equation}
for some constants $M\ge 1$ and $\omega \in {\mathbb R}$.
 \item[(ii)] If
$\lambda_1(L+V)=-\infty$, then for any $0\le u_0\in
L^2_\mu\setminus \{0\},$ there is no nonnegative weak solution of $(P_V)$
satisfying \eqref{eq22}.
\end{enumerate}
\end{theorem}


%
%


The existence of positive solutions to $(P_V)$
is related to Hardy's inequality on the
weighted space $L^2_\mu$. The nonexistence of solutions is due to the optimality of the constant in the Hardy inequality.
Therefore, studying the bottom of the spectrum is equivalent to studying
the weighted Hardy inequality
\begin{equation}\label{whi}
\quad c\int_{\R^N}\frac{\varphi^2}{|x|^2}\,d\mu\leq\int_{\R^N}|\nabla \varphi |^2\,d\mu+C\int_{\R^N} \varphi^2\,d\mu,\quad \varphi\in H^1_\mu,
\end{equation}
and the sharpness of the best constant possible.

\noindent A special case
is given when $\mu_A(x)=Ke^{-\frac{1}{2}\langle Ax,x\rangle}$,
where $A$ is a positive real Hermitian $N\times N$ matrix
and $K$ is a normalizing constant. The operator
$L$ becomes the well-known symmetric Ornstein-Uhlenbeck operator $Lu=\Delta u-Ax\cdot \nabla u$. Using this approach
 it was proved in \cite{gol-gol-rhan} that if $\mu_A(x)=Ke^{-\frac{1}{2}\langle Ax,x\rangle}$,
then there exists an exponentially bounded weak solution $u\in C([0,\infty),L^2_{\mu_A})$
provided that $0\leq c \leq c_0(N)$, and no positive exponentially bounded weak solution exists if $c>c_0(N)$.

The result was generalized in \cite{gol-gol-rhan2}  for the density measure
$\mu(x)=Ke^{-\sigma(x) }$
with  $c_1|x|^2\leq \sigma(x) \leq c_2|x|^2$.
Furthermore,
under more general hypotheses on $d\mu$
the argument was extended to a larger class of  Kolmogorov operators, including
the case $\sigma(x)\sim |x|^m$ with $m>2$
and $c\leq c_\sigma$, for a suitable constant $c_{\sigma}$ which is not  the optimal one.

In this paper we generalize these results for a larger class of measures $d\mu$,
including the case $\mu(x)=Ke^{-\frac{|x|^m}{m}}$ with $m>0$
and obtaining also the optimality of the constant.
We observe that such $\mu$ requires to satisfy more general hypotheses, which allow degeneracy at one point.
Indeed, for $0<m<1$ such a measure does not belong to $C^{1,\alpha}_{loc}(\R^N)$ since the gradient is not bounded in $0$.
Then, we will consider measures which are not necessarily $(1,\alpha)$-H\"olderian in the whole space but such that
$\mu\in C^{1,\alpha}_{loc}(\Omega )\ \textrm{where}\ \Omega:=\R^N\setminus\{0\}$.
Firstly,  we need
that the unperturbed operator $L$ generates a semigroup.
Hence, further  conditions on $\mu$ are required in order to guarantee generation results on $L^2_\mu$. We consider the following hypotheses.

\noindent {\bf Hypothesis $(H1)$:}
\begin{enumerate}
\item [i)]  $\mu\in C^{1,\alpha}_{loc}( \Omega)$;
 \item [ii)]
  $\mu\in H^{1}_{loc}(\R^N)$,  $\frac{\nabla \mu}{\mu}\in L^r_{loc}(\R^N)\ \textrm{for some}\ r>N$, and $\inf_{x\in K}\mu(x)>0$ for any compact $K\subset \R^N$.
 \end{enumerate}
By \cite[Corollary 3.7]{alb-lor-man} we have that
the closure of $(L,C_c^{\infty}(\R^N))$ on $L^2_\mu$
generates a strongly continuous Markov semigroup $T(t)$ on $L^2_\mu$,
which is also analytic.
Thanks to this result we can claim that, under Hypothesis $(H1)$,
Theorem \ref{thm11} holds.

The second step is, then, to obtain a weighted Hardy inequality.
To this purpose we observe that the operator $L+V$ in $L^2_\mu$ is equivalent to the Schr\"odinger operator
$H=\Delta +(U_\mu+V)$ in $L^2(\z)$, where
$$U_\mu :=\frac{1}{4}\left|\frac{\nabla \mu}{\mu}\right|^2-\frac{1}{2}\frac{\Delta \mu}{\mu}.$$
Indeed, taking the transformation $T\varphi=\frac{1}{\sqrt{\mu}}\varphi$  we have  $L+V=THT^{-1}$.
Now, roughly speaking, if $V=\frac{c}{|x|^2}$ we expect Hardy's inequality to hold if
 $U_\mu+\frac{c}{|x|^2}\leq \frac{c_0(N)}{|x|^2}$  in a neighbourhood of the origin,
that is $c\leq c_0(N)-|x|^2U_\mu$.


Thus, we consider the following  hypothesis on $\mu (x)$.

\noindent {\bf Hypothesis $(H2)$:}
\begin{enumerate}
 \item [i)]  $\mu\geq 0$, $\mu^{1/2}\in H^{1}_{loc}(\R^N)$, $\Delta \mu\in L^1_{loc}(\R^N)$;
\item [ii)] the constant
$$c_{0,\mu}:=\liminf_{x\to 0}\left( c_0(N)-|x|^2U_\mu\right)$$ is finite;
 \item  [iii)] for every $R>0$ the function
\begin{align*}
 &U:=U_\mu-\frac{1}{|x|^2}\limsup_{x\to 0}|x|^2U_\mu
\end{align*}
  is bounded from above in $\R^N\setminus B_R$;
\item [iv)] there exists a $R_0>0$ such that
  \[
  |x|^2U(x)\leq \frac{1}{4}\frac{1}{|\log|x||^2}, \quad \forall x\in B_{R_0}.
 \]
\end{enumerate}
Under the assumption $(H2)$ we obtain the weighted Hardy inequality
\eqref{whi} for any $c\leq c_{0,\mu}$.
If condition ${\rm iv)}$ of $(H2)$ is not satisfied we still obtain the weighted Hardy inequality if we only assume

\noindent {\bf Hypothesis $(H2')$: }
\begin{itemize}
\item[$\bullet$] Conditions $\rm i),ii),iii)$ of $(H2)$ hold.
\end{itemize}
In this case the constant $c_{0,\mu}$ is not achieved and we obtain \eqref{whi} for any $c<c_{0,\mu}.$

As regards the optimality, we consider the following hypothesis.

\noindent {\bf Hypothesis $(H3)$: }
\begin{itemize}
 \item [i)]$\mu\in L^1_{loc}(\R^N);$
 \item  [ii)] There exists $\sup_{\delta\in \R}\big\{ \frac{1}{|x|^\delta} \in L^1_{loc}(\R^N,d\mu)\big\}=:N_0.$
\end{itemize}
Under condition $(H3)$ Hardy's inequality does not hold if $c>c_0(N_0)=\left( \frac{N_0-2}{2}\right)^{2}$.
If, instead,  we have

\noindent {\bf Hypothesis $(H3')$: } Conditions $i)$ and $ii)$ of $(H3)$ hold and
\begin{itemize}
 \item [iii)]
 $$\limsup_{\lambda\to 0^+}\lambda\int_{B_1}|x|^{\lambda-N_0}\,d\mu=+\infty ,  $$
\end{itemize}
then the inequality does not hold if $c\geq c_0(N_0)$.

It is obvious that we have the best result when  $(H2)$ and $(H3)$ (respectively $(H2')$ and $(H3')$)
hold and the constant $c_{0,\mu}$ coincides with the constant $c_0(N_0)$.

Therefore, we can state our main results.
\begin{theorem}\label{th:whi-reached} Assume assumptions $(H2)$ and $(H3)$ and $c_{0,\mu}=c_0(N_0)$.
Then the weighted Hardy inequality \eqref{whi}  holds if and only if $c\leq c_0(N_0)$.
\end{theorem}

\begin{theorem}\label{th:whi-no-reached}
Assume assumptions $(H2')$ and $(H3')$ with $N_0>2$ and $c_{0,\mu}=c_0(N_0)$. Then
\eqref{whi} holds if and only if $c<c_0(N_0)$.
 \end{theorem}

Finally, putting together the weighted Hardy inequality \eqref{whi} and Theorem \ref{thm11},
we can state the following.
\begin{theorem}\label{thm-main2}
Assume that  hypotheses $(H1)$, $(H2)$  and $(H3)$ hold with $c_{0,\mu}=c_0(N_0)$
and $0\le V(x)\le \frac{c}{|x|^2}$,
then
the following assertions are satisfied:
\begin{enumerate}
\item[(i)] If $0\le c\le c_0(N_0)=\left( \frac{N_0-2}{2} \right)^2$, then there exists a weak
solution $u\in C([0,\infty),L^2_\mu)$ of $(P_V)$
satisfying
\begin{equation}\label{eq:est}
\|u(t)\|_{L^2_\mu}\le Me^{\omega t}\|u_0\|_{L^2_\mu},\quad
t\ge 0
\end{equation}
for some constants $M\ge 1$, $\omega \in {\mathbb R}$, and any $u_0\in L^2_\mu$. \item[(ii)] If
$c> c_0(N_0)$, then for any $0\le u_0\in L^2_\mu,\,u_0\neq 0,$
there is no positive weak solution of $(P_V)$
with $V(x)=\frac{c}{|x|^2}$ satisfying
(\ref{eq:est}).
\end{enumerate}
\end{theorem}

If instead, assumptions $(H2')$ and $(H3')$ with $N_0>2$ are fulfilled, the same statement holds but
the constant $c_0(N_0)$ is not achieved.

These hypotheses on the measure $\mu$ allow us to treat the case
\begin{equation*}
\mu(x)=ke^{-b{|x|^m}}
\end{equation*}
for $b,m>0$, associated to the operator
\begin{equation*}
Au=\Delta u-bm|x|^{m-2}x\cdot \nabla u+ \frac{c}{|x|^2}u,
\end{equation*}
since one can see that $\mu$ satisfies  $(H1)$, $(H2)$ and $(H3)$ with
constant $c_{0,\mu}=c_0(N_0)=c_0(N)$. Therefore, for such $\mu$, Theorem \ref{thm-main2} holds.

Moreover, under the same assumptions, one can also consider the measures
$\mu_{\beta}(x)=k\frac{1}{|x|^\beta}e^{-b|x|^m}$.
For such measures  Hypothesis $(H1)$ is not fulfilled, however we obtain the weighted Hardy inequality with best constant.
Indeed, also in this case, the constant of $(H2)$ coincides with the  best constant of $(H3)$ and it depends
upon the parameter $\beta$ with $\beta<N-2$.
We have explicitly $N_0=N-\beta$ and
$c_{0,\mu}=c_0(N_0)=\left( \frac{N-\beta-2}{2} \right)^2$.

For  $\mu(x)=\frac{1}{|x|^\beta }$ 
we recover the well known Caffarelli-Nirenberg inequality
\begin{equation*}
\left( \frac{N-\beta-2}{2} \right)^2
\int_{\R^N}\frac{\varphi^2}{|x|^2}|x|^{-\beta}dx\leq\int_{\R^N}|\nabla \varphi |^2
|x|^{-\beta}\,dx,\quad \varphi\in H^1(\R^N),
\quad \beta<N-2.
\end{equation*}


For  measures behaving like a logarithm near the origin
$$\mu \sim \left( \log \frac{1}{|x|} \right)^{\alpha},$$
we obtain \eqref{whi} with constant $c_0(N_0)=c_0(N)$.
If $\alpha \leq 0$ the constant is achieved and it is the best one.
Indeed, $\mu$ satisfies $(H1)$, $(H2)$ and $(H3)$ provided that $\alpha \leq 0$. If instead $\alpha>0$ then $\mu$ satisfies $(H1)$, $(H2')$ and $(H3')$. So, the constant is not achieved, but it still is the best one.

Finally, we also provide an example in which the constant in \eqref{whi} of $(H2')$ does not coincide with the optimal one of $(H3)$.


\section{Weak solutions and bottom of the spectrum}
In this section we prove that, under condition $(H1)$ on $\mu$, Theorem \ref{thm11} holds. Firstly we observe that $C_c^{\infty}(\R^N)$ is dense in $H^1_\mu$. This is given by the condition $\mu^{1/2}\in H^{1}_{loc}(\R^N)$ (see \cite[Theorem 1.1]{tolle-2012} ) which is ensured by
Hypothesis $ii)$ of $(H1)$. Indeed, $\mu\in H^{1}_{loc}(\R^N)$ implies   $\mu\in L^1_{loc}(\R^N)$ and $\nabla\mu \in L^2_{loc}(\R^N)$. Moreover, $\frac{\nabla \mu}{\mu}\in L_{loc}^r(\R^N)$ implies
$\frac{\nabla \mu}{\mu}\in L_{loc}^2(\R^N)$. Then $\mu^{\frac{1}{2}}\in L^2_{loc}(\R^N)$ and
$$\int_K \left|\nabla \mu^{\frac{1}{2}} \right|^2\,dx=\int_K\frac{1}{4}\frac{|\nabla \mu|^2}{\mu}\,dx
\leq \frac{1}{4}\left( 	\int_K\left|\frac{\nabla \mu}{\mu}\right|^2\,dx \right)^{\frac{1}{2}}\left( \int_K \left|\nabla \mu\right|^2\, dx \right)^{\frac{1}{2}}<\infty,$$
for every compact set $K\subset \R^N$. 
Moreover, one also obtains  that  $C_c^{\infty}\left( \Omega \right)$ is dense in $H^1_{\mu}$ by Hypothesis $ii)$ of $(H1)$ (see  the Appendix).

Now, we precise the definition of weak solutions.
Let us recall the problem
$$(P_{V})\quad \left\{\begin{array}{ll}
\partial_tu(x,t)=Lu(x,t)+V(x)u(x,t),\quad t>0,\,x\in {\mathbb R}^N,\\
u(\cdot ,0)=u_0\geq 0\in L^2_\mu.
\end{array}
\right. $$
We say that $u$ is a weak solution to $(P_V)$ if, for each
$T,\,R>0$, we have
$$u\in C([0,T],L^2_\mu),\, Vu\in L^1(B_R\times (0,T),d\mu \,dt) \hbox{\ and }$$
\begin{equation}\label{eq:weak-equation}
\int_0^T\int_{{\mathbb R}^N}u(-\partial_t\phi -L\phi)\,d\mu\,dt-\int_{{\mathbb R}^N}u_0\phi(\cdot ,0)\,d\mu=\int_0^T\int_{{\mathbb R}^N}Vu\phi
 \,d\mu \,dt
\end{equation}
 for all $\phi \in W_{loc}^{2,1}(Q_T)$ having compact support with $\phi(\cdot ,T)=0$,
where
$B_R$ denotes the open ball of ${\mathbb R}^N$ of radius $R$ and center $0$ and
for $T>0$, we write $Q_T$ for ${\mathbb R}^N\times [0,T]$.

Theorem \ref{thm11} is based on Cabr\'e-Martel's idea and it was proved
in \cite[Theorem 2.1]{gol-gol-rhan} for measures $\mu$ belonging to $C^{1,\alpha}_{loc}(\R^N)$.
The proof relies on certain  properties of the operator $L$ and its corresponding semigroup $T(t)$ in $L^2_\mu$.
Furthermore, the strict positivity on compact sets of $T(t)u_0$, if $0\leq u_0\in L^2_\mu \setminus\{0\}$ is required.

Hence, in order to claim that Theorem \ref{thm11} holds in our situation, we only have to ensure that these properties hold. This is stated in Proposition \ref{pr:lor-H0} and Lemma \ref{lm:strictly-positivity} below.

We recall that the measure $\mu(dx)$ is the infinitesimally invariant measure for the operator $L$, i.e.
\begin{equation*}
 \int L\varphi \,d\mu=0\,\text{ for every } \varphi\in C_c^{\infty}(\R^N).
\end{equation*}
Moreover, the operator $L$ is symmetric on $C_c^{\infty}(\R^N)$, i.e.
\begin{equation}\label{eq:int-part-mu}
\int_{\R^N} Lu \overline{v}\, d\mu=-\int \nabla u\cdot \nabla \overline{v}\,d\mu,\quad\text{ for every }u,v \in C_c^\infty(\R^N).
\end{equation}
Hence, by \cite[Corollary 3.7]{alb-lor-man}, we have that
the closure of $(L,C_c^{\infty}(\R^N))$ on $L^2_\mu$
generates a strongly continuous Markov semigroup $T(\cdot)$ on $L^2_\mu$,
which is also analytic.

Let  $(L,D(L))$ be the self-adjoint operator defined by the closure of $(L,C_c^{\infty}(\R^N))$ on $L^2_\mu$.

\begin{proposition}\label{pr:lor-H0} The following assertions hold.
\begin{itemize}
 \item [i)] $D(L)\subset H^1_{\mu}$.
 \item [ii)] For every $f\in D(L),\, g\in H^1_{\mu}$ we have
  \[\int L f \overline{g} \,d\mu=-\int \nabla f\cdot \nabla \overline{g}\,d\mu.\]
 \item [iii)] $T(t)L^2_\mu\subset D(L)$ for all $t>0$.
  \end{itemize}
\end{proposition}
\begin{proof} i) and ii).
Let $u\in D(L)$. Then there exists $u_n\in C_c^{\infty}(\R^N)$ such that $u_n\to u$ in $L^2_\mu$ and $Lu_n\to Lu$ in $L^2_\mu.$
By \eqref{eq:int-part-mu} we have
\begin{align*}
 \|\nabla u_n-\nabla u_m\|^2_{\mu}&=-\int (\overline{u_n}-\overline{u_m})(Lu_n-Lu_m)\,d\mu \\&\leq \|u_n-u_m\|_{\mu}\|Lu_n-Lu_m\|_\mu\;.
\end{align*}
Then, $\nabla u_n$ converges to a function $G\in \left(L_\mu^2\right)^N$.
On the other hand, one has
\[
 -\int Lu_n\overline{\varphi} \,d\mu=\int \nabla u_n\cdot \nabla \overline{\varphi}\, d\mu
\]
for every $\varphi \in C_c^{\infty}(\R^N)$. Taking the limit for $n\to \infty$, since $Lu_n\to Lu$ and $\nabla u_n\to G$, we have
\[
 -\int Lu \overline{\varphi}\, d\mu=\int G\cdot \nabla \overline{\varphi}\, d\mu.
\]
By a density argument, this holds true for any $\varphi \in H^1_\mu.$

It remains  to show that the components  $G_i$ of $G$ are the weak derivatives of $u.$
Fixing $\psi \in C_c^{\infty}(\R^N)$, one obtains
\[
 \int \partial _i u_n\overline{\psi}\, dx=-\int u_n \partial_i \overline{\psi}\, dx.
\]
Then, taking the limit as $n\to \infty$, one has
$
 \int G _i \overline{\psi}\, dx=-\int u \partial_i \overline{\psi}\, dx.
$

Indeed,
\begin{align*}
 \left|\int \partial _i u_n\overline{\psi}\, dx-\int G_i \overline{\psi}\, dx \right|&\leq \int \left|  G_i-\partial _iu_n \right| |\psi|\,dx\\
& \leq C_\psi\int \left|  G_i-\partial _iu_n \right| |\psi|\,d\mu\to 0\text{ as } n\to \infty,
\end{align*}
where $\frac {1}{C_\psi}=\inf _{x\in supp \psi}\mu(x).$
Similarly we have $\lim_{n\to \infty}\int u_n \partial_i \overline{\psi}\, dx=\int u \partial_i \overline{\psi}\, dx$.

Assertion iii) follows from the analyticity of the semigroup $T(\cdot)$
(cf. \cite[Theorem II.4.6]{EN00}, \cite[Lemma 1.3.3]{fuk-osh-tak}).\\
\end{proof}

Since $C_c^{\infty}(\Omega)$ is dense in $H^1_\mu$, where $\Omega =\R^N\setminus \{0\}$, we can prove the following Lemma.
\begin{lemma}\label{lm:strictly-positivity}
Let $V$ be a nonnegative function belonging to $L^1_{loc}(\R^N)$. Let $u$ be a nonnegative weak solution of $(P_V)$. Then, for every compact set $K\subset \Omega$ and $t>0$
there exists $c(t)>0$ (not depending on $V)$
such that $$u(x,t)\geq c(t)\int _{K}u_0\, d\mu ,\quad (x,t) \in K\times [0,T].$$
\end{lemma}
\begin{proof} Let $u_0 \in C_c^{\infty}(\R^N)$ nonnegative and let $u$ be a nonnegative weak solution of $(P_V)$.
Let
$C_R=B_R\setminus \overline B_{1/R}$
such that $K\subset C_R$ and let $\varphi\in C_c^{\infty}(C_R)$ such that $0\leq \varphi\leq 1$.

Consider the problem
\[(Pb)\,
 \left\{
 \begin{array}{ll}
 v_t(x,t)=Lv(x,t),& \text{ on }C_R\times (0,T],\\
 v(x,t)=0,&\text{ on } \partial C_R,\\
 v(x,0)=\varphi u_0.
 \end{array}
 \right.
\]
Then Problem $(Pb)$ admits a solution
$v\in C^{2+\alpha,1+\frac{\alpha}{2}}(\overline C_R\times [0,T])$. Moreover,
\[
v(x,t)=\int_{C_R}G(t,x,y)v(y,0)\,dy,
\]
where $G$ is a strictly positive and continuous function on $(0,+\infty)\times C_R\times C_R$.
Let $c(t)=\min_{(x,y)\in K\times K} G(t,x,y)$. We have for every $x\in K$
\[
 v(x,t)\geq \int_{K}G(t,x,y)v(y,0)\,dy\geq c(t)\int_Kv(y,0)\,dy.
\]
Furthermore,  $v$ is a weak solution to $v_t=Lv$ in $C_R$. In particular,
for all $\phi \in W^{2,1}_2 (C_R\times [0,T])$ with $\phi(\cdot,0)\geq0$
having compact support with $\phi(\cdot, T)=0$, we have
$$\int_0^T\int_{C_R}v(-\partial_t\phi -L\phi)\,d\mu\,dt-\int_{C_R }(\varphi u_0)\phi(\cdot ,0)\,d\mu=0.$$
Comparing with \eqref{eq:weak-equation}, one obtains
\begin{equation}\label{eqA}
\int_0^T\int_{C_R}(v-u)(-\partial_t\phi -L\phi)\,d\mu\,dt=\int_{C_R }(\varphi u_0-u_0-Vu )\phi(\cdot ,0)\,d\mu\leq 0.
\end{equation}

Fix $T,\,R>0$, $0\le \psi
\in C_c^\infty(C_R\times [0,T])$ and consider the parabolic problem
$$\left\{\begin{array}{ll}
\partial_t\phi +L\phi =-\psi ,& \hbox{on }C_R\times (0,T),\\
\phi|_{\partial C_R\times (0,T)}= 0,\\
\phi(x,T)=0,& x\in C_R.
\end{array} \right.
$$
By \cite[Theorem IV.9.1]{lad-sol-ura}, one obtains a solution $0\le \phi \in W^{2,1}_2(C_R\times (0,T))$. By a standard argument, one
can insert the solution $\phi$ in (\ref{eqA}). Therefore,
$$\int_0^T\int_{C_R }(v-u)\psi \,d\mu \,dt \le 0$$ for all $0\le \psi \in C_c^\infty(C_R\times [0,T])$. Thus,
$$u\geq v\geq c(t)\int_{K}\varphi u_0\, d\mu.$$
Since the last inequality holds for any $\varphi \in C_c^\infty (C_R)$ one obtains
$$u\geq c(t)\int_{K}u_0\, d\mu.$$
\end{proof}

Therefore, we can state the following theorem, for which we omit the proof because it is similar to that of Theorem \ref{thm11}, see \cite[Theorem 2.1]{gol-gol-rhan}.

\begin{theorem}\label{th:cabr-mart-H1}  Assume that $\mu$ satisfies Hypothesis $(H1)$. Let $0\le V\in
L_{loc}^1({\mathbb R}^N)$. Then the following hold:
\begin{enumerate}
\item[(i)] If $\lambda_1(L+V)>-\infty$, then there exists a
nonnegative weak solution \\$u\in C([0,\infty),L^2_\mu)$ of $(P_V)$
satisfying
\begin{equation}\label{eq222}
\|u(t)\|_{L^2_\mu}\le Me^{\omega t}\|u_0\|_{L^2_\mu},\quad t\ge
0
\end{equation}
for some constants $M\ge 1$ and $\omega \in {\mathbb R}$.
 \item[(ii)] If
$\lambda_1(L+V)=-\infty$, then for any $0\le u_0\in
L^2_\mu\setminus \{0\},$ there exists no nonnegative weak solution of $(P_V)$
satisfying \eqref{eq222}.
\end{enumerate}
\end{theorem}

\section{The Hardy inequality}


Let  $d\mu$ be a positive measure (not necessary a probability measure) with density $\mu(x)$.
Let us recall the definition of $c_{0,\mu}$ and the potential $U$. We
set
$$U_\mu :=\frac{1}{4}\left|\frac{\nabla \mu}{\mu}\right|^2-\frac{1}{2}\frac{\Delta \mu}{\mu}$$
and
$$c_{0,\mu}:=\liminf_{x\to 0}\left( c_0(N)-|x|^2U_\mu\right).$$
Consider
\begin{align*}
 &U:=U_\mu-\frac{1}{|x|^2}\limsup_{x\to 0}|x|^2U_\mu .
\end{align*}
So, we have
\begin{align}\label{u}
 &U=U_\mu+\frac{c_{0,\mu}}{|x|^2}-\frac{c_0(N)}{|x|^2}.
\end{align}

We start by proving of the following improved Hardy inequality.

\begin{proposition}\label{pr:hardy-gen00}
Assume $\rm i),ii)$ of $(H2)$. Then, for any $\varphi\in C_c^\infty(\R^N)$, the following inequality holds
\begin{equation}\label{eq:hardy-gen00-1}
c_{0,\mu}\int_{\R^N}\frac{\varphi^2}{|x|^2}d\mu\leq\int_{\R^N}|\nabla \varphi |^2\,d\mu
+\int_{\R^N} U\varphi^2\,d\mu .
\end{equation}
\end{proposition}
\begin{proof}
One has
\begin{align*}
\varphi (x)\sqrt {\mu (x)}&=-\int_1^\infty \frac{d}{dt}\left( \varphi(tx)\sqrt{\mu(tx)} \right)\,dt\\
&=-\int_1^\infty x\cdot \left( \nabla \varphi(tx)+\frac{1}{2}\varphi(tx)\frac{\nabla \mu(tx)}{\mu(tx)} \right)\sqrt {\mu(tx)}\,dt.
\end{align*}
By Minkowski's inequality for integrals and by a change of variables we have
\begin{align*}
\left\|\frac{\varphi }{|x|} \right\|_{L^2_\mu}&=
\left\|\frac{\varphi \sqrt {\mu (x)}}{|x|} \right\|_{L^2}\\
&\leq \left\|
\int_1^\infty \left|
  \frac{x}{|x|}\cdot \left( \nabla \varphi(tx)+\frac{1}{2}\varphi(tx)\frac{\nabla \mu(tx)}{\mu(tx)} \right)
    \sqrt{ \mu(tx)}
    \right|
    \,dt
\right\|_{L^2}\\
& \leq
\int_1^\infty
\left\|
\frac{x}{|x|}\cdot \left( \nabla \varphi(tx)+\frac{1}{2}\varphi(tx)\frac{\nabla \mu(tx)}{\mu(tx)} \right)\sqrt{\mu(tx)}
\right\|_{L^2}\,
dt\\
& =\int_1^\infty
  \left(\int_{\R^N}
  \left| \nabla\varphi(tx)+\frac{1}{2}\varphi(tx)\frac{\nabla \mu(tx)}{\mu(tx)}\right|^2\mu(tx)\,
  dx\right)^{\frac{1}{2}}\,
  dt\\
& =\int_1^\infty t^{-\frac{N}{2}}\,dt
  \left\|\nabla \varphi+\frac{1}{2}\varphi\frac{\nabla \mu}{\mu}\right\|_{L^2_{\mu}}\\
& =\frac{1}{\sqrt {c_0(N)}}
  \left\|\nabla \varphi+\frac{1}{2}\varphi\frac{\nabla \mu}{\mu}\right\|_{L^2_{\mu}}.
\end{align*}

Hence,
\begin{align*}
c_0(N)\int _{\R^N}\frac{\varphi^2}{|x|^2}\,d\mu &\le \int _{\R^N}|\nabla\varphi |^2\,d\mu
    +\frac{1}{4}\int_{\R^N}\varphi ^2  \left|\frac{\nabla\mu}{\mu} \right|^2\,d\mu
     +\int_{\R^N}\nabla\varphi\cdot\nabla\mu\frac{\varphi}{\mu}\,d\mu\\
&=\int _{\R^N}|\nabla\varphi |^2\,d\mu
    +\frac{1}{4}\int_{\R^N}\varphi ^2  \left|\frac{\nabla\mu}{\mu} \right|^2\,d\mu
    +\int_{\R^N}\frac{1}{2}\nabla\varphi^2 \cdot\nabla\mu\,  dx\\
&=\int _{\R^N}|\nabla\varphi |^2\,d\mu
    +\frac{1}{4}\int_{\R^N}\varphi ^2  \left|\frac{\nabla\mu}{\mu} \right|^2\,d\mu
    -\frac{1}{2}\int_{\R^N}\varphi^2 \Delta \mu\,  dx\\
&=\int _{\R^N}|\nabla\varphi |^2\,d\mu
    +\frac{1}{4}\int_{\R^N}\varphi ^2  \left|\frac{\nabla\mu}{\mu} \right|^2\,d\mu
    -\frac{1}{2}\int_{\R^N}\varphi^2 \frac{\Delta \mu}{\mu}  \,d\mu\\
&=\int _{\R^N}|\nabla\varphi |^2\,d\mu
    +\int_{\R^N}U_\mu\varphi ^2 \, d\mu.\\
\end{align*}
Then, \eqref{eq:hardy-gen00-1} follows from the relation
\[
U_\mu=U+\frac{c_0(N)-c_{0,\mu}}{|x|^2}.
\]
\end{proof}

We recall, under assumption $(i)$ of $(H2)$, that $C_c^\infty(\z)$ is dense in $H_\mu^1$.
If moreover $U$ is bounded from above, the result below is a direct consequence
of (\ref{eq:hardy-gen00-1}).

\begin{corollary}\label{pr:hardy-gen-bound0}
Assume $\rm i), ii)$ of $(H2)$ and assume that there
exists $C_{\mu}\in \R$ such that $U\leq C_\mu$.
Then, for any $\varphi\in H^1_\mu$,
\begin{equation*}
c_{0,\mu}\int_{\R^N}\frac{\varphi^2}{|x|^2}\,d\mu\leq\int_{\R^N}|\nabla \varphi |^2\,d\mu+
{C_\mu}\int_{\R^N} \varphi^2\,d\mu.
\end{equation*}
\end{corollary}
If instead of the boundedness from above of $U$ in $\R^N$, we assume $U$ bounded only for $|x|$ large enough
(that is $(H2')$),
then the following result holds.
\begin{theorem}\label{pr:h20}
Let us assume that Hypothesis  $(H2')$ holds,
then for every $c<c_{0,\mu}$ there exists $C_\mu$ such that
for any $\varphi\in H^1_\mu$ the weighted Hardy inequality holds
\begin{equation}\label{in1}
c\int_{\R^N}\frac{\varphi^2}{|x|^2}\,d\mu\leq\int_{\R^N}|\nabla \varphi |^2
d\mu+{C_\mu}\int_{\R^N} \varphi^2\,d\mu.
\end{equation}
\end{theorem}
\begin{proof}
Since $\limsup_{x\to 0}|x|^2U=0$, it follows that for every $\varepsilon>0$
there exists $R_\varepsilon>0$ such that
$U\leq \frac{\varepsilon}{|x|^2}$ for all
$x\in B_{R_\varepsilon}$. Moreover, there exists $C_\mu$ depending on $R_\varepsilon$
such that $U\leq C_\mu$ for every $x\in B^c_{R_\varepsilon}.$
Then, by Proposition \ref{pr:hardy-gen00}, we have
\begin{align*}
c_{0,\mu}\int_{\R^N}\frac{\varphi^2}{|x|^2}\,d\mu
 & \leq\int_{\R^N}|\nabla \varphi |^2\,d\mu+\int_{\R^N} U\varphi^2\,d\mu\\
&  \leq\int_{\R^N}|\nabla \varphi |^2\,d\mu+\varepsilon\int_{B_{R_\varepsilon}}
\frac{\varphi^2}{|x|^2}\,d\mu
    +C_\mu\int_{B^c_{R_\varepsilon}} \varphi^2\,d\mu\\
& \leq\int_{\R^N}|\nabla \varphi |^2\,d\mu
    +\varepsilon\int_{R^N} \frac{\varphi^2}{|x|^2}\,d\mu+C_\mu\int_{\R^N} \varphi^2\,d\mu.
\end{align*}
The result follows by taking $c=c_{0,\mu}-\varepsilon$.
\end{proof}

We look now for weaker conditions with respect to the boundedness from above for $U$  in $\R^N$ in order to get (\ref{in1}). To this purpose, we have to consider improved Hardy's inequalities.

The first step is to state a relation between the weighted Hardy inequality
and a special improved Hardy's inequality.

\begin{lemma}\label{th:equivalence-hary}
Assume $\rm i), ii)$ of $(H2)$, and the improved Hardy inequality
 \begin{equation}\label{eq:2}
 c_0(N)\int \frac{u^2}{|x|^2}\,dx-\int |\nabla u |^2\,dx+\int U(x)u^2\,dx\leq C_\mu\int_{\R^N} u^2\,dx, \  u\in C_c^\infty(\R^N).
\end{equation}
Then, the weighted Hardy inequality holds
 \begin{equation}\label{whic}
c_{0,\mu}\int_{\R^N}\frac{\varphi^2}{|x|^2}\,d\mu\leq\int_{\R^N}|\nabla \varphi |^2
\,d\mu+{C_\mu}\int_{\R^N} \varphi^2\,d\mu,
 \  \varphi \in C_c^\infty(\R^N).
\end{equation}

\end{lemma}
\begin{proof} Let $\varphi\in C_c^{\infty}(\R^N)$ and set $u:=\varphi\sqrt \mu\in H^1(\R^N)$ with compact support.
By  \eqref{eq:2}, which holds by density for such a function $u$, integrating by parts and recalling the expression \eqref{u} for $U$, one obtains
\begin{align*}
c_0(N)\int_{\R^N}\frac{u^2}{|x|^2}\,dx&=c_0(N)\int_{\R^N}\frac{\varphi^2}{|x|^2}\,d\mu\\
&\leq \int_{\R^N}|\nabla (\varphi\sqrt{\mu}) |^2\,dx-\int_\z U(x)\varphi^2\,d\mu+{C_\mu}\int_{\R^N} \varphi^2\,d\mu\\
&=\int_\z|\nabla\varphi|^2\,d\mu+\frac{1}{2}\int_\z\nabla\varphi^2\nabla{\mu}\,dx+\frac{1}{4}\int_\z\varphi^2\left|\frac{\nabla\mu}{\mu}\right|^2\,d\mu\\
&\quad-\int_\z U_\mu\varphi^2\,d\mu+c_0(N)\int_\z\frac{\varphi^2}{|x|^2}\,d\mu\\
&\quad-c_{0,\mu}\int_\z\frac{\varphi^2}{|x|^2}\,d\mu+{C_\mu}\int_{\R^N} \varphi^2\,d\mu\\
& =\int_{\R^N}\left|\nabla \varphi \right|^2\,d\mu+{C_\mu}\int_{\R^N} \varphi^2\,d\mu\\
&\quad +c_0(N)\int_\z\frac{\varphi^2}{|x|^2}\,d\mu-c_{0,\mu}\int_\z\frac{\varphi^2}{|x|^2}\,d\mu.
\end{align*}
Then, inequality \eqref{whic} follows.
\end{proof}

Now, our aim is to prove \eqref{eq:2}.
Brezis and V\'azquez in \cite{bre-vaz} proved the following inequality
\[
\int_D |\nabla u|^2\,dx\geq c_0\int_D \frac{u^2}{|x|^2}\,dx+K\|u\|^2_{L^q(D)}
\]
with $q<\frac{2N}{N-2}$, for a bounded domain $D$ and for every $u\in H^1_0(D)$. From this, by H\"older's inequality, it follows an inequality of type
\[
 \int_D |\nabla u|^2\,dx\geq  c_0\int_D \frac{u^2}{|x|^2}\,dx+b\int_D Uu^2\,dx
\]
with potential $U$ belonging to $L_{loc}^p(\R^N)$ for $p>\frac{N}{2}$.
This gives us the desired result \eqref{whic},
but forces us  to suppose $U_\mu \in L_{loc}^{N/2+\varepsilon}(\R^N)$.

Therefore, in order to prove \eqref{eq:2}, we will refer to the following improved Hardy inequality,
see \cite{mus}, \cite{adi-cha-ram}.

\bigskip

\begin{theorem}\label{th:improved-hardy}
For any $u \in C_c^{\infty}(B_1)$ the following inequality holds
\begin{equation*}
\int_{B_1}\left|\nabla u\right|^2\,dx\geq c_0\int_{B_1}
\frac{u^2}{|x|^2}\,dx+\frac{1}{4}\int_{B_1}\frac{u^2}{|x|^2|\log|x||^2}\,dx.
\end{equation*}
\end{theorem}
Now, we suppose that $\mu$ and $U$ satisfy condition $(H2)$. We finally obtain the weighted Hardy inequality \eqref{whic}.

\begin{theorem}\label{thm:3-6}
Assume that Hypothesis $(H2)$ holds.
Then for any $\varphi\in H^1_\mu$,
the following inequality holds
\begin{equation*}
c_{0,\mu}\int_{\R^N}\frac{\varphi^2}{|x|^2}\,d\mu\leq\int_{\R^N}
|\nabla \varphi |^2\,d\mu+{C_\mu}\int \varphi^2\,d\mu.
\end{equation*}
\end{theorem}
\begin{proof}
By Proposition \ref{th:equivalence-hary}, we need to prove that
\begin{equation*}
 c_0(N)\int_{\R^N}\frac{u^2}{|x|^2}\,dx-\int_{\R^N} |\nabla u |^2\,dx+\int_{\R^N}U(x)u^2\,dx\leq C_\mu\int_{\R^N} u^2\,dx,\,\,
 \forall u\in C_c^\infty(\R^N).
\end{equation*}
By Hypothesis
$(H2)$ on $U$, there exists a $R_0\leq 1$ (otherwise one takes $R_0=1$) such that $U\leq
\frac{1}{4}\frac{1}{|x|^2|\log|x||^2}$ in $B_{R_0}$. Then, for $u\in C_c^\infty(\R^N)$, by
Theorem \ref{th:improved-hardy} and a change of variables, one has
\begin{align}\label{eq:musR}
\int_{B_{R_0}}Uu^2\,dx&\leq\frac{1}{4}\int_{B_{R_0}} \frac{u^2}{|x|^2|\log|x||^2}\,dx\nonumber\\
&\leq\frac{R_0^{N-2}}{4}\int_{B_1}\frac{u^2(R_0y)}{|y|^2|\log|y||^2}\,dy\nonumber\\
&\leq R_0^{N}\int_{B_1}|\nabla u(R_0y)|^2\,dy-R_0^{N-2}c_0(N)\int_{B_1}\frac{u^2(R_0y)}{|y|^2}\,dy\nonumber\\
&=\int_{B_{R_0}}|\nabla u|^2 \,dx-c_0(N)\int_{B_{R_0}}\frac{u^2}{|x|^2}\,dx.
\end{align}
Let $u\in C_c^{\infty}(\R^N)$ and  $\theta\in C^\infty(\z)$, $0\leq\theta\leq1$, such that
$\theta=1$ in $B_{\frac{R_0}{2}}$ e $\theta=0$ in $B^c_{R_0}$, and
$\frac{1}{2}\frac{|\nabla \theta|^2}{\theta}-\Delta \theta\leq M$.
Note that such a function exists. For instance, one can consider a translation and a dilatation of the function $\theta_0(s)=ce^{-\frac{1}{1-s^2}}$ for $|s|\leq1$ and equals to 0 for $|s|\geq1$.

Therefore, by \eqref{eq:musR} and Hypotesis $(H2)$, one obtains
\begin{align*}
c_0(N)\int_{\R^N} \frac{u^2}{|x|^2}\,dx&+\int_{\R^N} Uu^2\,dx\\
&\quad =
c_0(N)\int_{B_{R_0}} \frac{u^2}{|x|^2}\theta\, dx+\int_{B_{R_0}} Uu^2\theta\, dx\\
&\qquad +c_0(N)\int_{B^c_{\frac{R_0}{2}}} \frac{u^2}{|x|^2}(1-\theta)\,dx+\int_{B^c_{\frac{R_0}{2}}} Uu^2(1-\theta)\,dx\\
&\quad \leq \int_{B_{R_0}} \left|\nabla \left(u \sqrt \theta\right)\right|^2\,dx
    +\left(\frac{c_0(N)}{(R_0/2)^2}+K  \right)\int_{B^c_{\frac{R_0}{2}}} u^2\,dx\\
&\quad \leq \int_{\R^N} |\nabla u|^2\theta^2\,dx+\int_{\R^N} u^2
    \left( \frac{1}{4}\frac{|\nabla \theta|^2}{\theta}-\frac{1}{2}
\Delta \theta  \right) \,dx\\&\qquad+C\int_{\R^N} u^2\,dx\\
&\quad \leq \int_{\R^N} |\nabla u|^2\,dx+
  \int_{\R^N} u^2\left( \frac{1}{4}\frac{|\nabla \theta|^2}{\theta}-
\frac{1}{2}\Delta \theta \right)\,dx\\&\qquad
  +C\int_{\R^N} u^2\, dx\\
&\quad \leq \int_{\R^N} |\nabla u|^2\,dx+(\frac{M}{2}+C)\int_{\R^N} u^2\, dx.
\end{align*}
\end{proof}


\begin{remark}
For $N_e\in \R$ define
$c_0(N_e):=\left( \frac{N_e-2}{2} \right)^2.$ Take $c_{0,\mu}=c_0(N_e)$ if $c_{0,\mu}\geq 0$ and $N_e=2$ if $c_{0,\mu}<0$. Then, by Theorem \ref{thm:3-6} (resp. Theorem \ref{pr:h20}),
 Inequality \eqref{whi} with constant $c\leq c_0(N_e)$
(resp. $c<c_0(N_e)$) holds provided that $(H2)$ (respectively $(H2')$) is satisfied.
\end{remark}

\section{Optimality of the constant in the weighted Hardy inequality}
In this section we give conditions in order to prove the sharpness of the constant $c$ in \eqref{whi}.

\begin{theorem}
Let us assume Hypothesis $(H3)$.  Then, there exists a function in $H^1_\mu$ for which the weighted Hardy inequality \eqref{whi} does not  hold if $c>c_0(N_0)$.
\end{theorem}
\begin{proof}
Let $\gamma$ be such that $\max\{-\sqrt c,-\frac{N_0}{2}\}< \gamma
<\min\{ \frac{-N_0+2}{2},0\}$ so that $\gamma^2<c$ and from the definition of $N_0$ it follows that $|x|^{2\gamma}\in L^1_{loc}(\R^N,d\mu)$
and $|x|^{2\gamma-2}\notin L^1_{loc}(\R^N,d\mu)$.

Let $n\in\N$ and $\vartheta\in C_c^\infty(\R^N)$, $0\leq \vartheta\leq 1$,
$\vartheta=1$ in $B_1$ and $\vartheta=0$ in $B_2^c$.
Set $\varphi_n(x)=\min\{|x|^\gamma \vartheta(x),n^{-\gamma}\}$.
We observe that
\[
\varphi_n(x)=\left\{
\begin{array}{cl}
n^{-\gamma} &\text{ if }|x|<\frac{1}{n},\\
|x|^\gamma&\text{ if }\frac{1}{n}\leq |x|< 1,\\
|x|^\gamma\vartheta(x) &\text{ if }1\leq |x|<2,\\
0&\text{ if }|x|\geq 2.\\
\end{array}
\right.
\]
The functions $\varphi_n$ are in $H^1_\mu$.

Let us assume that $c>c_0(N_0)$, then we have to prove that the bottom of the spectrum of the operator $-\Delta +
\frac{\nabla \mu}{\mu}\cdot \nabla-\frac{c}{|x|^2}$
\[
\lambda_1=\inf_{\varphi \in H^1_\mu\setminus \{0\} }
\left(
  \frac{\int_{\R^N} \left( |\nabla \varphi|^2-\frac{c}{|x|^2}
\varphi^2 \right)\, d\mu}{\int_{\R^N} \varphi ^{2}\,d\mu}
\right)
\]
is $-\infty$.
To this purpose we have
\begin{align}\label{eq:stima0}
\int_{\R^N} \left( |\nabla \varphi_n|^2-\frac{c}{|x|^2}\varphi_n^2 \right)\, d\mu
& =\int_{B_1\setminus B_{\frac{1}{n}}} \left( \left|\nabla |x|^\gamma \right|^2-
\frac{c}{|x|^2}|x|^{2\gamma} \right)\, d\mu\nonumber\\
&\quad+
\int_{B_1^c} \left( |\nabla |x|^\gamma \vartheta(x) |^2-\frac{c}{|x|^2}
\left( |x|^\gamma \vartheta(x) \right)^2 \right)\, d\mu\nonumber\\
&\quad -c \int_{B_{\frac{1}{n}}} n^{-2\gamma}\frac{1}{|x|^2}\,d\mu\nonumber\\
& \leq (\gamma^2-c)\int_{B_1\setminus B_{\frac{1}{n}}} |x|^{2\gamma-2}\, d\mu+
2\int_{B_1^c} |x|^{2\gamma} |\nabla \vartheta |^2\,d\mu
    \nonumber\\&\quad+2\gamma^2\int_{B_1^c}\vartheta^2|x|^{2\gamma-2}\,d\mu\nonumber\\
& \leq (\gamma^2-c)\int_{B_1\setminus B_{\frac{1}{n}}} |x|^{2\gamma-2}\, d\mu
  +(2\|\nabla \vartheta \|_\infty^2+2\gamma^2)\int_{B_1^c}\,d\mu\nonumber\\
& =(\gamma^2-c)\int_{B_1\setminus B_{\frac{1}{n}}} |x|^{2\gamma-2}\, d\mu+C_1\;.
\end{align}
On the other hand,
\begin{align}\label{eq:stima2}
&\int_{\R^N} \varphi_n ^{2}\,d\mu \geq  \int_{B_2\setminus B_1} |x|^{2\gamma}\vartheta^2(x)\,d\mu=C_2\;.
\end{align}
Taking into account  \eqref{eq:stima0} and \eqref{eq:stima2} we have
\[
\lambda_1\leq \frac{\int_{\R^N} \left( |\nabla \varphi_n|^2-\frac{c}{|x|^2}\varphi_n^2 \right)\, d\mu}
    {\int_{\R^N} \varphi_n ^{2}\,d\mu}
\leq \frac{(\gamma^2-c)\int_{B_1\setminus B_{\frac{1}{n}}} |x|^{2\gamma-2}\, d\mu+C_1}{C_2}\;.
\]
We observe that $\gamma^2-c<0$. Taking the limit $n\to \infty$ we get
\[
 \lim_{n\to \infty }\int_{B_1\setminus B_{\frac{1}{n}}}|x|^{2\gamma-2}\, d\mu=+\infty .
\]
Hence
$\lambda_1=-\infty$.
\end{proof}

Therefore we obtain the following result.

\begin{theorem}\label{th:hardy-achieved}
Assume hypotheses $(H2)$ and $(H3)$ with $c_{0,\mu}=c_0(N_0)$. Then
for any $\varphi\in H^1_\mu $ the following inequality holds
\begin{equation*}
c_0(N_0)\int_{\R^N}\frac{\varphi^2}{|x|^2}\,d\mu\leq\int_{\R^N}|\nabla \varphi |^2
\,d\mu+{C_\mu}\int_{\R^N} \varphi^2\,d\mu
\end{equation*}
and $c_0(N_0)$ is the best constant.
\end{theorem}


If Hypotesis $(H3')$ with $N_0>2$ holds, one obtains the following theorem.
\begin{theorem}
Let us assume Hypothesis $(H3')$ with $N_0>2$.  Then,  there exists a function in $H^1_\mu$ for which the weighted Hardy inequality \eqref{whi}
does not hold if $c=c_0(N_0)=\left(\frac{N_0-2}{2}\right)^2$.
\end{theorem}
\begin{proof}
Let $\gamma$ be such that $0>\gamma > \frac{-N_0+2}{2}$ so that
$|x|^{2\gamma}\in L^1_{loc}(\R^N,d\mu)$
and $|x|^{2\gamma-2}\in L^1_{loc}(\R^N,d\mu)$
and $\gamma^2<c_0(N_0)$. Let $\vartheta\in C_c^\infty(\R^N)$,
$0\leq \vartheta\leq 1$, $\vartheta=1$ in $B_1$ and $\vartheta=0$ in $B_2^c$.
Set $\varphi_\gamma(x)=|x|^\gamma \vartheta(x)$.
We observe that $\varphi_\gamma\in H^1_\mu$.

Let us set $c=c_0(N)$. We have to prove that
\[
\lambda_1=\inf_{\varphi \in H^1_\mu\setminus \{0\} }
\left(
  \frac{\int_{\R^N} \left( |\nabla \varphi|^2-\frac{c}{|x|^2}\varphi^2 \right)
\,d\mu}{\int_{\R^N } \varphi ^{2}\,d\mu}
\right)
\]
is $-\infty$.
One has
\begin{align}\label{eq:stima0.1}
\int_{\R^N} \left( |\nabla \varphi_\gamma|^2-
\frac{c}{|x|^2}\varphi_\gamma^2 \right)\, d\mu&\quad =\int_{B_1} \left( \left|\nabla |x|^\gamma \right|^2-
\frac{c}{|x|^2}|x|^{2\gamma} \right)\, d\mu\nonumber\\
&\quad+
\int_{B_1^c\cap B_2} \left( |\nabla |x|^\gamma \vartheta(x) |^2-
\frac{c}{|x|^2}\left( |x|^\gamma \vartheta(x) \right)^2 \right)\, d\mu\nonumber\\
& \leq (\gamma^2-c)\int_{B_1} |x|^{2\gamma-2}\, d\mu
    +2\int_{B_1^c\cap B_2} |x|^{2\gamma} |\nabla \vartheta |^2\,d\mu\nonumber\\&\quad+
2\gamma^2\int_{B_1^c\cap B_2}\vartheta^2|x|^{2\gamma-2}\,d\mu\nonumber\\
& \leq (\gamma^2-c)\int_{B_1} |x|^{2\gamma-2}\, d\mu
  +(2\|\nabla \vartheta \|_\infty^2+2\gamma^2)\int_{B_1^c\cap B_2}\,d\mu\nonumber\\
& =(\gamma^2-c)\int_{B_1} |x|^{2\gamma-2}\, d\mu+C_1\;.
\end{align}

On the other hand,
\begin{align}\label{eq:stima2.1}
&\int_{\R^N} \varphi_\gamma ^{2}\,d\mu \geq  \int_{B_1} |x|^{2\gamma}\,d\mu\geq
\int_{B_1}\, d\mu=C_2\;.
\end{align}
Taking into account  \eqref{eq:stima0.1} and \eqref{eq:stima2.1}, we have
\[
\lambda_1\leq \frac{\int_{\R^N} \left( |\nabla \varphi_\gamma|^2-
\frac{c}{|x|^2}\varphi_\gamma^2 \right)\, d\mu}
    {\int_{\R^N} \varphi_\gamma ^{2}\,d\mu}
\leq \frac{(\gamma^2-c)\int_{B_1} |x|^{2\gamma-2}\, d\mu+C_1}{C_2}\;.
\]
Now, taking the limit $\gamma \to \left(\frac{-N_0+2}{2}\right)^+$, by Hypothesis $(H3')$, one obtains
\begin{align*}
& \lim_{\gamma \to \left(\frac{-N_0+2}{2}\right)^+ }(\gamma^2-c)\int_{B_1} |x|^{2\gamma-2}\, d\mu\\
&\quad =\lim_{\gamma \to \left(\frac{-N_0+2}{2}\right)^+ }
  \left(\gamma+\frac{-N_0+2}{2}\right)\left(\gamma-\frac{-N_0+2}{2}\right)\int_{B_1} |x|^{2\gamma-2}\, d\mu
  \\
&\quad=\lim_{\lambda\to 0^+}(\lambda+2-N_0)\lambda\int_{B_1} |x|^{2\lambda-N_0} d\mu=-\infty.
 \end{align*}

Then
$\lambda_1=-\infty$.
\end{proof}

Therefore, we obtain the following result.
\begin{theorem}\label{th:hardy-no-achieved}
Assume hypotheses $(H2')$ and $(H3')$ with $N_0>2$ and $c_{0,\mu}=c_0(N_0)$. Then
for every $\varphi\in H^1_\mu$ the following inequality holds
\begin{equation*}
c\int_{\R^N}\frac{\varphi^2}{|x|^2}\,d\mu<\int_{\R^N}|\nabla \varphi |^2
\,d\mu+{C_\mu}\int \varphi^2\,d\mu
\end{equation*}
for all $c<c_0(N_0)$
and $c_0(N_0)$ is the best constant.
\end{theorem}

In conclusion, we have proved
Theorem \ref{th:hardy-achieved} and \ref{th:hardy-no-achieved}  which, together with Theorem \ref{th:cabr-mart-H1}, give
Theorem \ref{thm-main2}.

\section{Examples}
We finally give some examples of measures for which the weighted Hardy inequality holds.

\begin{proposition}\label{prop-fed1}
Let $d\mu=\rho(x) dx$ with $\rho(x)=e^{-b|x|^{m}}$, $m>0,\,b\ge 0,$ and $N\ge 3$. Then
there exists a positive constant $C$ such that for all $u\in H^1_\mu$ the following inequality
\begin{equation*}
c_0(N)\int_{\R^N}\frac{u^2}{|x|^2}\,d\mu\leq\int_{\R^N}|\nabla u |^2\,
d\mu+C\int_{\R^N}u^2\,d\mu\;.
\end{equation*}
holds with best constant.
\end{proposition}
\begin{proof}
This measure satisfies assumptions $i)$ and $ii)$ of $(H2)$. Then, by a simple computation one obtains
\[
U_\mu=-\frac{1}{4}b^2m^2|x|^{2m-2}+\frac{1}{2}bm(N+m-2)|x|^{m-2},
\]
and $\lim_{x\to 0}|x|^2U_\mu(x)=0$. Therefore, $c_{0,\mu}=c_0(N)$, $U_\mu=U$ is a bounded from above far from $0$ and assumptions $(H2)$ and $(H3)$ are satisfied with $N_0=N$. Then the assertion follows from Theorem \ref{th:hardy-achieved}. Moreover, $\mu$ satisfies also Hypothesis $(H1)$.
\end{proof}

\begin{proposition}\label{prop-fed2}
Let us consider
$d\mu=\frac{1}{|x|^\beta}\rho(x)dx$ with $\rho(x)$ as in Proposition \ref{prop-fed1}, $N\ge 3$ and $\beta<N-2$.
 Then there exists a constant $C\ge 0$ such that for all $u\in H^1_\mu$
\begin{equation*}
c_0(N-\beta)\int_{\R^N}\frac{u^2}{|x|^2}\,d\mu
\leq\int_{\R^N}|\nabla u |^2
\,d\mu+C\int_{\R^N}u^2\,d\mu
\end{equation*}
holds with best constant.
\end{proposition}
\begin{proof}
This measure satisfies assumptions $i)$ and $ii)$ of $(H2)$ if $\beta<N-2$. Moreover, by a simple computation, one has $c_{0,\mu}=c_0(N-\beta)$ and
\[
 U=-\frac{1}{4}b^2m^2|x|^{2m-2}+\frac{1}{2}bm(N+m-2-\beta)|x|^{m-2}.
 \]
Then  $(H2)$ and $(H3)$ hold with $c_{0,\mu}=c_0(N-\beta)$, $N_0=N-\beta$. The result follows from Theorem \ref{th:hardy-achieved}.
\end{proof}


By taking $b=0$ in Proposition \ref{prop-fed2} we can also state the Caffarelli-Nirenberg inequality.
\begin{corollary}
If $\beta <N-2$ and $N\ge 3$ then the following inequality holds
\begin{equation*}
\left( \frac{N-2-\beta}{2} \right)^2\int_{\R^N}\frac{u^2}{|x|^2}|x|^{-\beta}\,dx
\leq\int_{\R^N}|\nabla u |^2|x|^{-\beta}\,dx,
\qquad \forall u\in H^1(\R^N).
\end{equation*}
\end{corollary}

The following is an example for a weight which behaves like logarithm near $0$.

\begin{proposition}
Let $\theta \in C_c^\infty(\z)$ with $0\le \theta \le 1$ such that $\theta =1$ on $B_{1/2}$ and $\theta=0$ on $B_1^c$. Assume that
$$\mu(x)= \theta(x)\left( \log \frac{1}{|x|} \right)^{\alpha},\quad x\in \z.$$
\begin{itemize}
\item[$\bullet$] If $\alpha \leq 0$, then there exists a constant $C\ge 0$ such that for all $u\in H^1_\mu$
\begin{equation*}
c_0(N)\int_{\R^N}\frac{u^2}{|x|^2}\,d\mu
\leq\int_{\R^N}|\nabla u |^2
\,d\mu+C\int_{\R^N}u^2\,d\mu
\end{equation*}
holds with best constant.
\item[$\bullet$] If $\alpha >0$, then there exist $c,C\ge 0$ such that for all $u\in H^1_\mu$
\begin{equation*}
c\int_{\R^N}\frac{u^2}{|x|^2}\,d\mu
\leq\int_{\R^N}|\nabla u |^2
\,d\mu+C\int_{\R^N}u^2\,d\mu
\end{equation*}
holds for every $c<c_0(N)$ and $c_0(N)$ is the best constant.
\end{itemize}
\end{proposition}
\begin{proof}
One has that $\mu$ satisfies assumptions  $(H2')$ and $(H3)$ with $c_{0,\mu}=c_0(N)$ and $N_0=N$ for all $\alpha\in\R$. Moreover, $\mu$ satisfies $(H2)$ if and only if $\alpha \leq  0$
and $\mu$ satisfies $(H3')$
if and only if $\alpha > 0$. Indeed, by a simple computation, one obtains
\begin{equation}\label{eq:esempio1}
|x|^2U_\mu=\left( \frac{1}{4}-\frac{(\alpha-1)^2}{4} \right)\left( \log\frac{1}{|x|} \right)^{-2}
  +\frac{\alpha}{2}(N-2)\left( \log\frac{1}{|x|} \right)^{-1},\quad x\in B_{1/2}.
\end{equation}
We have $\limsup_{x\to 0}|x|^2U_\mu=0$, then the constant in Hardy's inequality is $c_0(N)$. By \eqref{eq:esempio1}   it is easy to check that assumption $(H2)$ is satisfied if and only if $\alpha \leq 0.$ Then, for $\alpha\leq0$ Theorem \ref{th:hardy-achieved} applies.

Now, to show that $(H3')$ is satisfied for $\alpha >0$ it suffices to prove that
\[
 \lim_{\lambda\to 0^+}\lambda\int_{B_r}|x|^{\lambda-N}\,d\mu=+\infty
\]
for some positive $r<\frac{1}{2}$. To this purpose we have
\begin{align*}
&\gamma\int_{B_r}\frac{|x|^\lambda}{|x|^N}\,d\mu=\omega_N\int_0^r\lambda s^{\lambda-1}\left( \log \frac{1}{s} \right)^\alpha\, ds \\
&\quad  =\omega_N\int_0^rs^{\lambda-1}\alpha\left( \log \frac{1}{s} \right)^{\alpha-1}\,ds
  +\omega_N\left[s^\lambda\left( \log\frac{1}{s} \right)^{\alpha}\right]_0^r.
\end{align*}
The second term is uniformly bounded for every $\lambda>0$. As regards
the first term, it grows to infinity for $\lambda\to 0^+$ if and only if $\alpha-1>-1.$ Therefore, for $\alpha>0$, the assertion follows from Theorem \ref{th:hardy-no-achieved}.
\end{proof}

Finally, in the following example the constants $c_{0,\mu}$ and $c_0(N)$ do not coincide.

\begin{proposition}
Let $\mu$ be the density $\mu(x)=\mu_1\chi_{B_{1/2}}(x)+\mu_2\chi_{B^c_{1/2}}(x)$ where
$$\mu_1= (2+\sin\log |x|)$$
and
$\mu_2\in C^2_b(\R^N)$ a nonnegative smooth function such that $\mu\in H^2_{loc}(\R^N)$.
Then $\mu$ satisfies $(H1)$, $(H2')$ and $(H3)$ with $c_{0,\mu}<c_0(N)$.
\end{proposition}
\begin{proof} We have
\[
 |x|^2U_\mu(x)=\frac{1}{4}\left( \frac{\cos \log |x|}{2+\sin\log |x|} \right)^{2}
    -\frac{1}{2}\frac{(N-2)\cos \log |x|-\sin\log |x|}{2+\sin\log |x|},\quad x\in B_{1/2}.
\]
Moreover, it can be seen that $\limsup_{x\to 0}|x|^2U_\mu(x)>0$. Hence, $c_{0,\mu}<c_0(N)$. Since $\mu$ is bounded and positive, we have
\[
 \int_{B_r}|x|^{\gamma-N}\,d\mu<\infty
\]
for some $r>0$ and for every $\gamma >0$ and $\int_{B_r}\frac{1}{|x|^N}\,d\mu=\infty$.
Then $N_0=N.$
\end{proof}

\section{Appendix}
Let $d\mu$ be a positive measure and $\Omega :=\R^N\setminus \{0\}$. Set $L^p_\mu=L^p(\R^N,d\mu)$ and $\|u\|_{p,\mu}=\left( \int_{\R^N}|u|^p\,d\mu \right)^{\frac{1}{p}}$ for $p\in [1,\infty)$.

\begin{proposition}\label{prop:appendix}
Let $W_\mu^{1,p}=\overline{C_c^{\infty}(\R^N)}^{\|\cdot\|_{W_\mu^{1,p}}}$,
where $\|u\|_{W_\mu^{1,p}}=\|u\|_{p,\mu}+\|\nabla u\|_{p,\mu}$.
If
\begin{equation}\label{eq:cond1}
\lim_{\delta\to 0}\frac{1}{\delta^p}\int_{B_\delta}\,d\mu=0
\end{equation}
then $C_c^{\infty}(\Omega)$ is dense in $W^{1,p}_{\mu}$.
\end{proposition}
\begin{proof}
Let $u\in C_c^\infty(\R^N)$. We have to approximate $u$ with functions in $C_c^{\infty}(\Omega)$ with respect to the norm $\| \cdot \|_{W_\mu^{1,p}}$.

Let $\vartheta\in C_b^{\infty}(\R^N)$ such that $\chi_{B^c_1}\leq \vartheta\leq \chi_{B^c_{\frac{1}{2}}}$, and set
$\vartheta_n(x)=\vartheta(nx)$. We observe that
$\vartheta_n\to 1$ pointwisely in $\Omega$ and
$\|\nabla \vartheta_n\|_{\infty}\leq Cn$.
We have
\begin{align*}
\|u-(u\vartheta_n)\|^{p}_{W_\mu^{1,p}}\leq
C\left(\|u(1-\vartheta_n)\|^{p}_{p,\mu}+\|\nabla \left( u(1-\vartheta_n)\right)\|^{p}_{p,\mu}\right).
\end{align*}
The first term of the right hand side converges to $0$ by dominated convergence. As regards the second one we have
\begin{align*}
\|\nabla u(1-\vartheta_n)\|^{p}_{p,\mu}
  &\leq
  C\left(\int_{\R^N}(1-\vartheta_n)^p|\nabla u|^p\,d\mu +\int_{\R^N}|\nabla \vartheta_n|^p | u|^p\,d\mu\right)\\
&\leq C\left( \int_{\R^N}(1-\vartheta_n)^p|\nabla u|^p\,d\mu +n^p\int_{B_{\frac{1}{n}}}|   u|^p\,d\mu\right)\\
&\leq C\left(\int_{\R^N}(1-\vartheta_n)^p|\nabla u|^p\,d\mu +n^p\|u\|_{\infty}^p\int_{B_{\frac{1}{n}}} \,d\mu\right).\\
\end{align*}
So, the first integral converges to $0$ by dominated convergence, the last one by Condition \eqref{eq:cond1}. Thus,
$$\lim_{n\to \infty}\|u-(u\vartheta_n)\|^{p}_{W_\mu^{1,p}}=0.$$

\end{proof}
\begin{corollary}
Let $p<N$. If $\mu \in W^{1,p}_{\loc}\left( \R^N \right)$ then
$C_c^{\infty}\left( \Omega \right)$ is dense in $W_\mu^{1,p}$.
\end{corollary}
\begin{proof}
By Proposition \ref{prop:appendix}, it suffices to verify  Condition \eqref{eq:cond1}. Let us observe first that
since $\mu \in W^{1,p}_{\loc}\left( \R^N \right)$, by Sobolev's embedding theorem, it follows that
$\mu\in L_\loc^{p^*}(\R^N)$ with $p^*=\frac{Np}{N-p}$.\\
Then,
\begin{align*}
\frac{1}{\delta^p}\int_{B_\delta}\mu\, dx
\leq \frac{1}{\delta^p}\left( \int_{B_\delta}\mu^{p^*}\,dx \right)^{\frac{1}{p^*}}\left( \int_{B_\delta}\,dx \right)^{\frac{p-1}{p}+\frac{1}{N}}
\leq C\delta ^{\frac{N(p-1)}{p}+1-p}.
\end{align*}
One can easily verify that $\frac{N(p-1)}{p}+1-p>0$ if $p<N.$
\end{proof}

\bibliography{bibfile-hardy,bibfile}{}
\bibliographystyle{amsplain}

\end{document}